\documentclass[11pt]{amsart}

\usepackage{amssymb,amscd,amsmath,hyperref,color,enumerate}
\usepackage[all,cmtip]{xy}
\usepackage{mathrsfs}
\usepackage{lipsum}
\usepackage{nicematrix,tikz, ifthen}
\usepackage{scalerel}
\usepackage{ascii}


\usepackage[margin=3cm]{geometry}

\newcommand{\N}{{\ensuremath{\mathbb{N}}}}
\newcommand{\Z}{{\ensuremath{\mathbb{Z}}}}
\newcommand{\Q}{{\ensuremath{\mathbb{Q}}}}
\newcommand{\C}{{\ensuremath{\mathbb{C}}}}

\newcommand{\F}{{\ensuremath{\mathbb{F}}}}
\newcommand{\K}{{\ensuremath{\mathbb{K}}}}
\newcommand{\BB}{{\ensuremath{\mathrm{B}}}}

\usepackage[normalem]{ulem}

\newcommand{\stkout}[1]{\ifmmode\text{\sout{\ensuremath{#1}}}\else\sout{#1}\fi}

\DeclareMathOperator{\supp}{supp}

\DeclareMathOperator{\chr}{char}

\DeclareMathOperator{\idem}{Idem}

\newcommand{\ca}[1]{\ensuremath{\mathcal{#1}}}

\newtheorem{proposition}{Proposition}[section]
\newtheorem{lemma}[proposition]{Lemma}

\newtheorem{theorem}[proposition]{Theorem}

\newtheorem{corollary}[proposition]{Corollary}
\theoremstyle{definition}

\newtheorem{example}[proposition]{Example}

\numberwithin{equation}{section}

\newlength{\leftstackrelawd}
\newlength{\leftstackrelbwd}
\def\leftstackrel#1#2{\settowidth{\leftstackrelawd}%
	{${{}^{#1}}$}\settowidth{\leftstackrelbwd}{$#2$}%
	\addtolength{\leftstackrelawd}{-\leftstackrelbwd}%
	\leavevmode\ifthenelse{\lengthtest{\leftstackrelawd>0pt}}%
	{\kern-.5\leftstackrelawd}{}\mathrel{\mathop{#2}\limits^{#1}}}

\begin{document}
	
	\title[Classification of Jordan multiplicative maps on matrix algebras]{Classification of Jordan multiplicative maps on matrix algebras}
	
	\author{Ilja Gogi\'{c}, Mateo Toma\v{s}evi\'{c}}
	
	\address{I.~Gogi\'c, Department of Mathematics, Faculty of Science, University of Zagreb, Bijeni\v{c}ka 30, 10000 Zagreb, Croatia}
	\email{ilja@math.hr}
	
	\address{M.~Toma\v{s}evi\'c, Department of Mathematics, Faculty of Science, University of Zagreb, Bijeni\v{c}ka 30, 10000 Zagreb, Croatia}
	\email{mateo.tomasevic@math.hr}
	
	\thanks{We thank the referees for reviewing the manuscript
		and for their comments and suggestions.}
	

	
	\date{\today}
	
	\begin{abstract}
		Let $M_n(\mathbb{F})$ be the algebra  of $n \times n$ matrices over a field $\mathbb{F}$ of characteristic not equal to  $2$. If $n\ge 2$, we show that an arbitrary map $\phi : M_n(\mathbb{F}) \to M_n(\mathbb{F})$ is Jordan multiplicative, i.e.\ it satisfies the functional equation
		$$
		\phi(XY+YX)=\phi(X)\phi(Y)+\phi(Y)\phi(X), \quad \text{for all } X,Y \in M_n(\mathbb{F})
		$$
		if and only if one of the following holds: either $\phi$ is constant, equal to $P/2$ for some idempotent $P \in M_n(\mathbb{F})$, or there exists an invertible matrix $T \in M_n(\mathbb{F})$ and a ring monomorphism $\omega: \mathbb{F} \to \mathbb{F}$ such that  
		$$
		\phi(X)=T\omega(X)T^{-1} \quad \text{ or } \quad \phi(X)=T\omega(X)^tT^{-1}, \quad \text{for all } X \in M_n(\mathbb{F}),
		$$
		where $\omega(X)$ denotes the matrix obtained by applying $\omega$ entrywise to $X$. In particular, any Jordan multiplicative map $\phi : M_n(\mathbb{F}) \to M_n(\mathbb{F})$ with $\phi(0)=0$ is automatically additive. The analogous characterization fails when $\mathbb{F}$ has characteristic $2$. 
	\end{abstract}
	
	\maketitle
	
	\section{Introduction}
	An interesting class of problems in algebra revolves around exploring the interaction between the multiplicative and additive structures of rings and algebras. A landmark result in this area, due to Martindale \cite[Corollary]{Martindale}, states that any bijective multiplicative map from a prime ring containing a nontrivial idempotent onto an arbitrary ring is necessarily additive, and thus a ring isomorphism. Another fundamental result by Jodeit and Lam in \cite{JodeitLam} provides a classification of non-degenerate multiplicative self-maps on the matrix rings $M_n(\ca{R})$ over a principle ideal domain $\ca{R}$ (i.e.\ maps that are not identically zero on all zero-determinant matrices).  Specifically, they show that for each such map $\phi:M_n(\ca{R})\to M_n(\ca{R})$  one of the following holds: either there exists a nonzero idempotent matrix $P \in M_n(\ca{R})$ such that the map $\phi-P$ is multiplicative and degenerate, or there exists an invertible matrix $T \in M_n(\ca{R})$ and a ring endomorphism $\omega$ of $\ca{R}$ such that
	$$
	\phi(X)=T\omega(X)T^{-1} \qquad \text{or} \qquad  \phi(X)=T\omega(X)^*T^{-1}, \quad  \text{for all }  X \in M_n(\ca{R}),
	$$
	where $\omega(X)$ denotes the matrix obtained by applying $\omega$ entrywise to $X$, and $(\cdot)^*$ represents the corresponding cofactor matrix.  In particular, all bijective multiplicative self-maps on $M_n(\ca{R})$ are automatically additive and, consequently, ring automorphisms of $M_n(\ca{R})$. More recently, \v{S}emrl in \cite{Semrl2} provided an extensive classification of (non-degenerate) multiplicative self-maps on matrix rings over arbitrary division rings. Additionally, in \cite{Semrl1}, \v{S}emrl described the structure of multiplicative bijective maps on standard operator algebras, which are subalgebras of bounded linear maps on a complex Banach space that contain all finite-rank operators.

	\smallskip
	
	On the other hand, any associative ring (algebra) $\ca{A}$ naturally inherits the structure of a Jordan ring (algebra), via the \emph{Jordan product} defined by
	$$
	x \diamond y:= xy+yx, \quad  \text{for all }  x,y \in \ca{A}.
	$$
	When working with algebras $\ca{A}$ over a field $\F$ of characteristic not equal to $2$,  it is often more convenient to use the \emph{normalized Jordan product}, defined by
	\begin{equation}\label{eq:normalizedJP}
		x \circ y:= \frac{1}{2}(xy+yx), \quad \text{for all }  x,y \in \ca{A}.
	\end{equation}
	The Jordan structure of algebras plays an important role in various areas, especially in the mathematical foundations of quantum mechanics (see e.g.\ \cite{Strocchi}). The corresponding  morphisms between rings (algebras) $\ca{A}$ and $\ca{B}$ are called \emph{Jordan homomorphisms}, which are additive (linear) maps $\psi : \ca{A} \to \ca{B}$ satisfying
	\begin{equation}\label{eq:Jordan_morphism_std}
		\phi(x \diamond y)=\phi(x) \diamond \phi(y), \quad  \text{for all }  x,y \in \ca{A}.
	\end{equation}
	For $2$-torsion-free rings (algebras), condition \eqref{eq:Jordan_morphism_std} is equivalent to the property that $\phi$ preserves squares, i.e.
	\begin{equation*}
		\phi(x^2) = \phi(x)^2, \quad  \text{for all }  x \in \ca{A},
	\end{equation*}
	and, trivially, to the condition
	\begin{equation}\label{eq:Jordan_morphism_norm}
		\phi(x \circ y)=\phi(x) \circ \phi(y), \quad \text{for all }  x,y \in \ca{A},
	\end{equation}
	when both $\ca{A}$ and $\ca{B}$ are $\F$-algebras with $\chr(\F)\neq 2$. The most notable examples of Jordan homomorphisms are multiplicative and antimultiplicative maps. In fact, one of the central problems in Jordan theory, initially addressed by Jacobson and Rickart in \cite{JacobsonRickart} (see also \cite{Herstein, Smiley})
	is to determine the conditions on rings (algebras) that guarantee that any (typically  surjective) Jordan homomorphism between rings (algebras) is either multiplicative, antimultiplicative, or, more generally, a suitable combination of such maps. For more recent developments on this topic, we refer to Bre\v sar's paper \cite{Bresar} and the references therein.
	
	\smallskip
	
	Furthermore, when both $\ca{A}$ and $\ca{B}$ are standard operator algebras, with $\dim \ca{A}>1$,  Molnar classifies all bijective maps $\phi : \ca{A} \to \ca{B}$ satisfying \eqref{eq:Jordan_morphism_norm} in \cite[Theorem~1]{Molnar}. An important consequence of Molnar's result is that all such maps are automatically additive. Moreover, the same classification result  applies to bijective maps $\phi$ satisfying  \eqref{eq:Jordan_morphism_std}. Indeed, as noted at the beginning of the proof of \cite[Theorem~3.1]{GogicTomasevic2}, if $\phi: \ca{A}\to \ca{B}$ is $\diamond$-preserving (where $\ca{A}$ and $\ca{B}$ are any $\F$-algebras over a field $\F$ with $\chr(\F)\ne 2$), then the map $\psi : \ca{A} \to \ca{B}$ defined by
	\begin{equation}\label{eq:supsitution Jmultiplicative}
		\psi(x):=2 \phi\left(\frac{x}{2}\right), \quad \text{for all } x \in \ca{A}
	\end{equation}
	is evidently  $\circ$-preserving. Referring back to Molnar’s classification theorem \cite[Theorem~1]{Molnar}, the finite-dimensional variant asserts that any bijective map $\phi : M_n(\C)\to M_n(\C)$, $n \ge 2$, satisfying \eqref{eq:Jordan_morphism_norm} (or \eqref{eq:Jordan_morphism_std}) takes the form
	$$
	\phi(X)=T\omega(X)T^{-1} \quad \text{ or } \quad \phi(X)=T\omega(X)^tT^{-1}, \quad  \text{for all } X \in M_n(\C),
	$$
	where $T\in M_n(\C)$ is an invertible matrix and $\omega$ is a ring automorphism of $\C$, with  $(\cdot)^t$ denoting matrix transposition. In our recent work \cite{GogicTomasevic2}, the authors extended both \cite[Corollary]{JodeitLam} and the finite-dimensional version of \cite[Theorem~1]{Molnar} to the context of injective maps on structural matrix algebras (SMAs), which are subalgebras of $M_n(\mathbb{C})$ containing all diagonal matrices (for a simple characterization of SMAs see \cite[Proposition 3.1]{GogicTomasevic1}). For additional variants and generalizations of Molnar's result, particularly those related to the automatic additivity of bijective maps satisfying \eqref{eq:Jordan_morphism_std} or \eqref{eq:Jordan_morphism_norm}, we refer the reader to \cite{Ji, JiLiu, LiXiao, Lu} and the references therein.
	
	\smallskip
	
	The objective of this paper is to present a complete classification of Jordan multiplicative self-maps on matrix algebras. In contrast to the more intricate Jodeit-Lam's classification of the corresponding multiplicative self-maps, the Jordan multiplicative case exhibits a notably simpler structure:
	\begin{theorem}\label{thm:main}
		Let $\mathbb{F}$ be a field with $\chr(\F)\ne 2$ and let $\phi : M_n(\F) \to M_n(\F)$, $n \ge 2$, be an arbitrary map  satisfying either  \eqref{eq:Jordan_morphism_std} or \eqref{eq:Jordan_morphism_norm}. Then, one of the following holds:
		\begin{itemize}
			\item[(a)] $\phi$ is constant, in which case there exists an idempotent $P \in M_n(\mathbb{F})$ such that $\phi \equiv P$ if $\phi$ is $\circ$-preserving, or
			$\phi \equiv P/2$ if $\phi$ is $\diamond$-preserving.
			\item[(b)] $\phi$ is an additive map, and thus a Jordan ring monomorphism of $M_n(\F)$. Consequently, there exists an invertible matrix $T \in M_n(\F)$ and a ring monomorphism $\omega:\F \to \F$ such that
			\begin{equation}\label{eq:form of phi}
				\phi(X)=T\omega(X)T^{-1} \quad \text{ or } \quad \phi(X)=T\omega(X)^tT^{-1}, \quad \text{for all } X \in M_n(\F).
			\end{equation}
		\end{itemize}
	\end{theorem}
	The proof of Theorem \ref{thm:main} will be presented in Section \S\ref{sec:main}. The approach follows a similar strategy to that in \cite{GogicTomasevic2}, relying entirely on elementary linear algebra techniques. Let us also highlight that a related variant of Theorem  \ref{thm:main}, concerning Jordan multiplicative self-maps on the real subspace of self-adjoint matrices in $M_n(\C)$ (with $n \ge 3$), was obtained by Fo\v sner et al.\ in \cite[Proposition~5.2]{Fosneretal}. We conclude the paper by demonstrating that non-constant Jordan multiplicative maps defined on general central SMAs, the $C^*$-algebra of bounded linear operators on an infinite-dimensional Hilbert space,
	or on $M_n(\F)$ when $\chr(\F)=2$ (the $\diamond$-multiplicative variant), are no longer automatically additive (Examples \ref{ex:notSMA}, \ref{ex:notB(H)} and \ref{ex:notchar2}). 
	
	\section{Notation and Preliminaries}\label{sec:prel}
	We now introduce some notation that will be used throughout the paper. Let $\F$ be a fixed field of characteristic not equal to $2$. By $\F^\times$ we denote the group of all nonzero elements in $\F$. Given a unital associative $\F$-algebra $\ca{A}$, by $\idem(\ca{A})$ we denote the partially ordered set of all idempotents $\ca{A}$, where
	$$
	p \le q \qquad \text{if} \qquad pq = qp = p.
	$$ 
	For $p \in \idem(\ca{A})$ we use the notation $p^\perp:=1-p \in \idem(\ca{A})$.  Further, for $p,q \in \idem(\ca{A})$ we write 
	$$
	p \perp q  \qquad \text{if} \qquad pq = qp =0.
	$$
	We use $\circ$ to denote the normalized Jordan product, defined by \eqref{eq:normalizedJP}. Obviously $p \in \ca{A}$ is an idempotent if and only if it is a Jordan idempotent (i.e.\ satisfies $p \circ p = p$). We explicitly state the following simple lemma from \cite{GogicTomasevic2}, which will be used on several occasions.
	\begin{lemma}[{\cite[Lemma~2.1]{GogicTomasevic2}}]\label{le:Jordan product calculations} Let $\ca{A}$ be an $\F$-algebra. For $p,q \in \idem(\ca{A})$ and an arbitrary $a \in \ca{A}$ we have:
		\begin{enumerate}[(a)]
			\item $p \circ a = 0$ if and only if $pa = ap = pap = 0$.
			\item $p \circ a = a$ if and only if $pa = ap = pap=a$.
			\item $p \perp q$ if and only if $p \circ q = 0$.
			\item $p \le q$ if and only if $p \circ q = p$.
		\end{enumerate}
	\end{lemma}
	
	\smallskip

	Let $n \in \N$ be a fixed positive integer.
	\begin{itemize}
		\item[--] By $[n]$ we denote the set $\{1,\ldots, n\}$ and by $\Delta_n$ the diagonal $\{(j,j) : j \in [n]\}$ in $[n]^2$.
		\item[--] By $M_n=M_n(\mathbb{F})$ we denote the algebra of $n \times n$ matrices over $\F$ and by  $\ca{D}_n$ its subalgebra consisting of all diagonal matrices. 
		\item[--] The rank of a matrix $X \in M_n$ is denoted by $r(X)$.
		\item[--] For matrices $X,Y \in M_n$, we write $X \propto Y$  to indicate that either  $X = Y = 0$, or they are both nonzero and collinear.
		\item[--] As usual, for $i,j \in [n]$, by $E_{ij}\in M_n$  we denote the standard matrix unit with $1$ at the position $(i,j)$ and $0$ elsewhere. For a matrix $X = [X_{ij}]_{i,j=1}^n \in M_n$ we define its \emph{support} as 
		$$
		\supp X:=\{(i,j)\in [n]^2 : X_{ij}\ne 0\}.
		$$ 
		\item[--] Given a ring endomorphism $\omega$ of $\F$, we use the same symbol $\omega$ to denote the induced ring endomorphism of $M_n$, defined by applying $\omega$ to each entry of the corresponding matrix:  
		$$\omega(X)=[\omega(X_{ij})]_{i,j=1}^n, \quad \text{for all } X=[X_{ij}]_{i,j=1}^n\in M_n.$$
	\end{itemize}
	
	It is well-known that $M_n$ is a simple algebra, and hence a simple Jordan algebra (see e.g. \cite[Corollary of Theorem~1.1]{Herstein-book}). In fact, we have the following simple yet useful observation.
	\begin{proposition}\label{p:M_n is Jordans simple}
		For an arbitrary matrix $X \in M_n$ define the subset $\ca{J}_X \subseteq M_n$ by
		$$
		\ca{J}_X:=\left\{(\cdots (X \circ Y_1)\circ Y_2) \circ \cdots )\circ Y_k : k \in \N, \, Y_1, \ldots, Y_k \in M_n\right\}.
		$$
		If $X \ne 0$, then $\ca{J}_X=M_n$. 
	\end{proposition}
	\begin{proof}
		Fix a nonzero matrix $X\in M_n$. It suffices to show that $I \in \ca{J}_X$. 
		
		\begin{itemize}
			\item Suppose that $X_{ij} \ne 0$ for some distinct $i,j \in [n]$. First of all, we have
			$$X_{ij}E_{ji} =E_{ji}X E_{ji} = (X \circ E_{ji}) \circ (2E_{ji})\in \ca{J}_X,$$
			so that
			$$E_{ii}=\left((X_{ij}E_{ji}) \circ \left(\frac2{X_{ij}}E_{ij}\right)\right)\circ E_{ii}  \in \ca{J}_X.$$
			
			\item Otherwise, suppose that $X \in \ca{D}_n$ and fix some $i \in [n]$ such that $X_{ii} \ne 0$. Then
			$$E_{ii} = \left(X \circ \left(\frac1{X_{ii}}E_{ii}\right)\right) \circ E_{ii} \in \ca{J}_X.$$
		\end{itemize}
		In any case, $E_{ii}\in \ca{J}_X$ for some $i \in [n]$ also implies $E_{jj}\in \ca{J}_X$ for all $j \in [n]$. Indeed, if  $j \ne i$:
		\begin{align*}
			E_{ji} = (2E_{ji})\circ E_{ii} \in \ca{J}_X & \implies  E_{ii}+E_{jj}=E_{ji}\circ (2E_{ij}) \in \ca{J}_X \\
			&\implies  E_{jj}=(E_{ii}+E_{jj}) \circ  E_{jj} \in \ca{J}_X.
		\end{align*}
		For $r \in [n]$ define
		$$D_r := \sum_{j \in [r]} E_{jj} \in M_n.$$
		We prove that $D_r \in \ca{J}_X$ for all $r \in [n]$ by induction on $r$ (then $I=D_n \in \ca{J}_X$). To illustrate the process on a concrete example, consider $n=5$. We have
		$$\underbrace{\begin{bmatrix}
				1 & 0 & 0 & 0 & 0\\
				0 & 0 & 0 & 0 & 0\\
				0 & 0 & 0 & 0 & 0\\
				0 & 0 & 0 & 0 & 0\\
				0 & 0 & 0 & 0 & 0\\
		\end{bmatrix}}_{=A_2\in \ca{J}_X} \circ \underbrace{\begin{bmatrix}
				0 & -4 & 0 & 0 & 0\\
				0 & 0 & 0 & 0 & 0\\
				0 & 0 & 0 & 0 & 0\\
				0 & 0 & 0 & 0 & 0\\
				0 & 0 & 0 & 0 & 0\\
		\end{bmatrix}}_{= B_2} = \underbrace{\begin{bmatrix}
				0 & -2 & 0 & 0 & 0\\
				0 & 0 & 0 & 0 & 0\\
				0 & 0 & 0 & 0 & 0\\
				0 & 0 & 0 & 0 & 0\\
				0 & 0 & 0 & 0 & 0\\
		\end{bmatrix}}_{\in \ca{J}_X},$$
		$$\underbrace{\begin{bmatrix}
				0 & -2 & 0 & 0 & 0\\
				0 & 0 & 0 & 0 & 0\\
				0 & 0 & 0 & 0 & 0\\
				0 & 0 & 0 & 0 & 0\\
				0 & 0 & 0 & 0 & 0\\
		\end{bmatrix}}_{\in \ca{J}_X} \circ \underbrace{\begin{bmatrix}
				0 & 0 & 0 & 0 & 0\\
				-1 & 0 & 0 & 0 & 0\\
				0 & 0 & 0 & 0 & 0\\
				0 & 0 & 0 & 0 & 0\\
				0 & 0 & 0 & 0 & 0\\
		\end{bmatrix}}_{= C_2} = \begin{bmatrix}
			1 & 0 & 0 & 0 & 0\\
			0 & 1 & 0 & 0 & 0\\
			0 & 0 & 0 & 0 & 0\\
			0 & 0 & 0 & 0 & 0\\
			0 & 0 & 0 & 0 & 0\\
		\end{bmatrix} = D_2 \implies D_2 \in \ca{J}_X,$$
		$$\underbrace{\begin{bmatrix}
				1 & 0 & 0 & 0 & 0\\
				0 & 1 & 0 & 0 & 0\\
				0 & 0 & 0 & 0 & 0\\
				0 & 0 & 0 & 0 & 0\\
				0 & 0 & 0 & 0 & 0\\
		\end{bmatrix}}_{=A_3\in \ca{J}_X} \circ \underbrace{\begin{bmatrix}
				0 & 0 & 4 & 0 & 0\\
				0 & -1 & 0 & 0 & 0\\
				0 & 0 & 0 & 0 & 0\\
				0 & 0 & 0 & 0 & 0\\
				0 & 0 & 0 & 0 & 0\\
		\end{bmatrix}}_{= B_3} = \underbrace{\begin{bmatrix}
				0 & 0 & 2 & 0 & 0\\
				0 & -1 & 0 & 0 & 0\\
				0 & 0 & 0 & 0 & 0\\
				0 & 0 & 0 & 0 & 0\\
				0 & 0 & 0 & 0 & 0\\
		\end{bmatrix}}_{\in \ca{J}_X},$$
		$$\underbrace{\begin{bmatrix}
				0 & 0 & 2 & 0 & 0\\
				0 & -1 & 0 & 0 & 0\\
				0 & 0 & 0 & 0 & 0\\
				0 & 0 & 0 & 0 & 0\\
				0 & 0 & 0 & 0 & 0\\
		\end{bmatrix}}_{\in \ca{J}_X} \circ \underbrace{\begin{bmatrix}
				0 & 0 & 0 & 0 & 0\\
				0 & -1 & 0 & 0 & 0\\
				1 & 0 & 0 & 0 & 0\\
				0 & 0 & 0 & 0 & 0\\
				0 & 0 & 0 & 0 & 0\\
		\end{bmatrix}}_{=C_3} = \begin{bmatrix}
			1 & 0 & 0 & 0 & 0\\
			0 & 1 & 0 & 0 & 0\\
			0 & 0 & 1 & 0 & 0\\
			0 & 0 & 0 & 0 & 0\\
			0 & 0 & 0 & 0 & 0\\
		\end{bmatrix} = D_3 \implies D_3 \in \ca{J}_X,$$
		$$\underbrace{\begin{bmatrix}
				1 & 0 & 0 & 0 & 0\\
				-2 & 1 & 0 & 0 & 0\\
				0 & 0 & 0 & 0 & 0\\
				0 & 0 & 0 & 0 & 0\\
				0 & 0 & 0 & 0 & 0\\
		\end{bmatrix}}_{= A_4 = A_4 \circ D_2 \in \ca{J}_X} \circ \underbrace{\begin{bmatrix}
				0 & 0 & 0 & 4 & 0\\
				0 & 0 & -4 & 8 & 0\\
				0 & 0 & 0 & 0 & 0\\
				0 & 0 & 0 & 0 & 0\\
				0 & 0 & 0 & 0 & 0\\
		\end{bmatrix}}_{=B_4} = \underbrace{\begin{bmatrix}
				0 & 0 & 0 & 2 & 0\\
				0 & 0 & -2 & 0 & 0\\
				0 & 0 & 0 & 0 & 0\\
				0 & 0 & 0 & 0 & 0\\
				0 & 0 & 0 & 0 & 0\\
		\end{bmatrix}}_{\in \ca{J}_X},$$
		$$\underbrace{\begin{bmatrix}
				0 & 0 & 0 & 2 & 0\\
				0 & 0 & -2 & 0 & 0\\
				0 & 0 & 0 & 0 & 0\\
				0 & 0 & 0 & 0 & 0\\
				0 & 0 & 0 & 0 & 0\\
		\end{bmatrix}}_{\in \ca{J}_X} \circ \underbrace{\begin{bmatrix}
				0 & 0 & 0 & 0 & 0\\
				0 & 0 & 0 & 0 & 0\\
				0 & -1 & 0 & 0 & 0\\
				1 & 0 & 0 & 0 & 0\\
				0 & 0 & 0 & 0 & 0\\
		\end{bmatrix}}_{=C_4} = \begin{bmatrix}
			1 & 0 & 0 & 0 & 0\\
			0 & 1 & 0 & 0 & 0\\
			0 & 0 & 1 & 0 & 0\\
			0 & 0 & 0 & 1 & 0\\
			0 & 0 & 0 & 0 & 0\\
		\end{bmatrix} = D_4 \implies D_4 \in \ca{J}_X,$$
		$$\underbrace{\begin{bmatrix}
				1 & 0 & 0 & 0 & 0\\
				-2 & 1 & 0 & 0 & 0\\
				0 & 0 & 1 & 0 & 0\\
				0 & 0 & 0 & 0 & 0\\
				0 & 0 & 0 & 0 & 0\\
		\end{bmatrix}}_{=A_5=A_5 \circ D_3 \in \ca{J}_X} \circ \underbrace{\begin{bmatrix}
				0 & 0 & 0 & 0 & 4\\
				0 & 0 & 0 & 4 & 8\\
				0 & 0 & -1 & 0 & 0\\
				0 & 0 & 0 & 0 & 0\\
				0 & 0 & 0 & 0 & 0\\
		\end{bmatrix}}_{=B_5} = \underbrace{\begin{bmatrix}
				0 & 0 & 0 & 0 & 2\\
				0 & 0 & 0 & 2 & 0\\
				0 & 0 & -1 & 0 & 0\\
				0 & 0 & 0 & 0 & 0\\
				0 & 0 & 0 & 0 & 0\\
		\end{bmatrix}}_{\in \ca{J}_X},$$
		$$\underbrace{\begin{bmatrix}
				0 & 0 & 0 & 0 & 2\\
				0 & 0 & 0 & 2 & 0\\
				0 & 0 & -1 & 0 & 0\\
				0 & 0 & 0 & 0 & 0\\
				0 & 0 & 0 & 0 & 0\\
		\end{bmatrix}}_{\in \ca{J}_X} \circ \underbrace{\begin{bmatrix}
				0 & 0 & 0 & 0 & 0\\
				0 & 0 & 0 & 0 & 0\\
				0 & 0 & -1 & 0 & 0\\
				0 & 1 & 0 & 0 & 0\\
				1 & 0 & 0 & 0 & 0\\
		\end{bmatrix}}_{=C_5} = \begin{bmatrix}
			1 & 0 & 0 & 0 & 0\\
			0 & 1 & 0 & 0 & 0\\
			0 & 0 & 1 & 0 & 0\\
			0 & 0 & 0 & 1 & 0\\
			0 & 0 & 0 & 0 & 1\\
		\end{bmatrix} = D_5=I \implies I \in \ca{J}_X.$$
		Of course, if $p:=\mathrm{char}(\F)\ne 0$, all operations are performed modulo $p$.
		
		\smallskip
		
		We continue with the proof for general $n$. For $r = 1$ we have $D_1=E_{11} \in \ca{J}_X$. Suppose that $D_{r-1} \in \ca{J}_X$ for some $2 \le r \le n$. We prove that $D_r\in \ca{J}_X$. Let $(p_j)_{j \in \N}$ be a sequence in $\F$ defined as
		$$p_j := \begin{cases}
			1, &\quad \text{ if }j = 1,\\
			2 \cdot 3^{j-2}, &\quad \text{ if }j \ge 2,\\
		\end{cases}$$
		and
		$$A_r := \begin{cases}
			\left(\sum_{1 \le j \le k} E_{jj}\right) - 2\left(\sum_{1 \le j < i \le k} E_{ij}\right), &\quad \text{ if }r = 2k,\\
			\left(\sum_{1 \le j \le k+1} E_{jj}\right) - 2\left(\sum_{1 \le j < i \le k} E_{ij}\right), &\quad \text{ if }r = 2k+1,\\
		\end{cases}$$
		$$B_r := \begin{cases}
			-8E_{k,k+1}+4\left(\sum_{1 \le i \le j \le k} p_i E_{j,2k+i-j}\right), &\quad \text{ if }r = 2k,\\
			-E_{k+1,k+1}+4\left(\sum_{1 \le i \le j \le k} p_i E_{j,2k+1+i-j}\right), &\quad \text{ if }r = 2k+1,\\
		\end{cases}$$
		$$C_r := \begin{cases}
			-E_{k+1,k}+\sum_{1 \le j \le k} E_{2k+1-j,j}, &\quad \text{ if }r = 2k,\\
			-E_{k+1,k+1}+\sum_{1 \le j \le k} E_{2k+2-j,j}, &\quad \text{ if }r = 2k+1.
		\end{cases}
		$$
		We have that $\supp A_r \subseteq [r-1]\times[r-1]$, so 
		$$A_r = A_r \circ D_{r-1} \in \ca{J}_X.$$
		Hence, using the observation $p_j = 2(p_1 + \cdots + p_{j-1})$, for $j \ge 2$, a straightforward calculation shows that
		$$D_r = (A_r \circ B_r) \circ C_r \in \ca{J}_X.$$
	\end{proof}
	We shall also require the following elementary fact, which is a simplified version of \cite[Lemma~3.3]{GogicTomasevic2} (applicable to general SMAs).
	\begin{lemma}\label{le:contains entire CxC}
		Let $\ca{S} \subseteq [n]^2 \setminus \Delta_n, n \ge 2$, be a nonempty subset. Suppose that for each $(i,j) \in \ca{S}$ we have:
		\begin{enumerate}[(a)]
			\item $(i,k) \in \ca{S}, \text{ for all } k \in [n]\setminus \{i\}$,
			\item $(l,j) \in \ca{S}, \text{ for all } l \in [n]\setminus \{j\}$,
			\item $(j,i) \in \ca{S}.$
		\end{enumerate}
		Then $\ca{S} = [n]^2 \setminus \Delta_n$.
	\end{lemma}
	\begin{proof}
		Fix some $(i,j) \in \ca{S}$ and let $(k,l) \in [n]^2 \setminus \Delta_n$ be arbitrary. If $k \ne j$, then
		$$(i,j) \in \ca{S} \stackrel{(b)}\implies (k,j) \in \ca{S} \stackrel{(a)}\implies (k,l) \in \ca{S}.$$
		If $l \ne i$, then
		$$(i,j) \in \ca{S} \stackrel{(a)}\implies (i,l) \in \ca{S} \stackrel{(b)}\implies (k,l) \in \ca{S}.$$
		Finally, if $(k,l) = (j,i)$, then the claim follows directly from (c).
	\end{proof}
	
		
	
	\section{Proof of Theorem \ref{thm:main}}\label{sec:main}
	Let $m,n\in \N$ be fixed throughout the proof. Before proving our main result, we first establish some preliminary results, starting with the following straightforward consequence of Proposition \ref{p:M_n is Jordans simple}.
	\begin{lemma}\label{le:zero map}
		Let $\ca{A}$ be an arbitrary $\F$-algebra and let  $\phi : M_n \to \ca{A}$ be a $\circ$-multiplicative map such that $\phi(X) = 0$ for some nonzero matrix $X \in M_n$. Then $\phi$ is the zero map.
	\end{lemma}
	\begin{proof}
		By Proposition \ref{p:M_n is Jordans simple} we have $\ca{J}_X=M_n$, and therefore $\phi(X)=0$ implies that $\phi$ is the zero map.  
	\end{proof}
	The following lemma, which is a variant of \cite[Lemma~3.4]{GogicTomasevic2} (originally for injective $\circ$-multiplicative maps on SMAs), outlines the general properties of (not necessarily injective)  $\circ$-multiplicative maps between matrix algebras.
	\begin{lemma}\label{le:basic properties II}
		Let $\phi : M_n \to M_m$, be a $\circ$-multiplicative map. Then the following holds true:
		\begin{enumerate}
			\item[(a)] $\phi$  preserves idempotents, i.e.\ $\phi(\idem(M_n))\subseteq \idem(M_m)$.
			\item[(b)] For $P,Q \in \idem(M_n)$ we have $P \le Q \implies \phi(P) \le \phi(Q)$.
		\end{enumerate}
		Suppose now that $\phi$ is nonzero but $\phi(0) = 0$. Then:
		\begin{enumerate}
			\item[(c)] For $P,Q \in \idem(M_n)$ we have $P \perp Q \implies \phi(P) \perp \phi(Q)$.
			\item[(d)] For each nonzero $P \in \idem(M_n)$ we have $r(\phi(P)) \ge r(P)$ (in particular, $m \ge n$).
		\end{enumerate}
		Further, if $m=n$, then:
		\begin{enumerate}
			\item[(e)] For each $P \in \idem(M_n)$ we have $r(\phi(P)) = r(P)$.
			\item[(f)] For each $P \in \idem(M_n)$ we have $\phi(P^\perp) = \phi(P)^\perp$.
			\item[(g)] The restriction $\phi|_{\idem(M_n)}: \idem(M_n)\to \idem(M_n)$ is orthoadditive, i.e.\
			$$P \perp Q \implies \phi(P+Q) = \phi(P)+\phi(Q), \quad  \text{for all } P,Q \in \idem(M_n).$$
			\item[(h)] Suppose that $P_1,\ldots,P_r \in \idem(M_n)$ are mutually orthogonal and let $\lambda_1,\ldots,\lambda_r \in \F$. Then
			$$\phi\left(\sum_{j=1}^r \lambda_j P_j\right) = \sum_{j=1}^r \phi(\lambda_j P_j).$$
		\end{enumerate}
	\end{lemma}
	\begin{proof}
		\begin{enumerate}[(a)]
			\item This is clear.
			\item We have $$\phi(P) = \phi(P \circ Q) = \phi(P) \circ \phi(Q)$$
			which is by Lemma \ref{le:Jordan product calculations} equivalent to  $\phi(P) \le \phi(Q)$.
			\item We have
			$$\phi(P) \circ \phi(Q) = \phi(P \circ Q) = \phi(0) = 0,$$
			so again by Lemma \ref{le:Jordan product calculations}, $\phi(P) \perp \phi(Q)$.
			\item Let $P \in \idem(M_n)$ be an arbitrary idempotent of rank $r \ge 1$. There exist mutually orthogonal rank-one idempotents $P_1, \ldots, P_r \in \idem(M_n)$ such that $P=P_1+\cdots + P_r$.  
			Since $\phi$ is not the zero map, by Lemma \ref{le:zero map} $\phi$ cannot annihilate any nonzero matrix, so in particular $\phi(P_j)\ne 0$ for all $j\in [r]$. Therefore, 
			$$P_1, \ldots, P_r \le P \stackrel{(b)}\implies \underbrace{\phi(P_1), \ldots, \phi(P_r)}_{\text{mutually orthogonal by (c)}}\le \phi(P).$$
			Consequently, $r(\phi(P)) \ge r$.
			\item Let $P\in \idem(M_n)$ be an arbitrary idempotent. By (c) we have $\phi(P) \perp \phi(P^\perp)$ and hence,
			$$n = r(P) + r(P^\perp) \stackrel{(d)}\le  r(\phi(P)) + r(\phi(P^\perp)) \le n$$
			and thus $r(\phi(P)) = r(P)$.
			\item In view of (c) and (e), we have that $\phi(P^\perp)$ is an idempotent orthogonal to $\phi(P)$ of rank $r(P^\perp)=r(\phi(P)^\perp)$. Consequently, $\phi(P^\perp) =\phi(P)^\perp$.
			
			\item Since $P\perp Q$, we have that $P+Q$ is again an idempotent and $P,Q \le P+Q$. Statements (b) and (c) imply $$\underbrace{\phi(P), \phi(Q)}_{\text{orthogonal}} \le \phi(P+Q)$$ and hence
			$$\phi(P)+\phi(Q) \le \phi(P+Q).$$
			Finally, we have
			\begin{align*}
				r(\phi(P) + \phi(Q)) &= r(\phi(P)) + r(\phi(Q)) \stackrel{(e)}= r(P) + r(Q) = r(P+Q) \\
				&\leftstackrel{(e)}= r(\phi(P+Q)),
			\end{align*}
			so equality follows.
			\item We have \begin{align*}
				\phi\left(\sum_{j=1}^r \lambda_j P_j\right) &= \phi\left(\left(\sum_{j=1}^r \lambda_j P_j\right) \circ \left(\sum_{l=1}^r  P_l\right)\right) = \phi\left(\sum_{j=1}^r \lambda_j P_j\right)  \circ \phi\left(\sum_{l=1}^r P_l\right)\\
				&\leftstackrel{(g)}= \phi\left(\sum_{j=1}^r \lambda_j P_j\right)  \circ \left(\sum_{l=1}^r \phi(P_l)\right) = \sum_{l=1}^r \left(\phi\left(\sum_{j=1}^r \lambda_j P_j\right)  \circ \phi(P_l)\right)\\
				&= \sum_{l=1}^r \phi\left(\left(\sum_{j=1}^r \lambda_j P_j\right)  \circ P_l\right) = \sum_{l=1}^r \phi(\lambda_l P_l).
			\end{align*}
		\end{enumerate}
	\end{proof}
	
		
	In the sequel, $\K$ will denote the prime subfield of $\F$, i.e.\ $\K$ is generated by the multiplicative identity of $\F$ (see e.g.\ \cite{DummitFoote}). Note that $\K \cong \Q$ if $\chr(\F)=0$, or $\K \cong \Z/p\Z$ if $p=\chr(\F)>0$.
	\begin{lemma}\label{le:h exists}
		Let $\phi : M_n \to M_n$ be a nonzero $\circ$-multiplicative map such that $\phi(0) = 0$. There exists a unique multiplicative map $\omega : \F \to \F$ such that 
		\begin{equation}\label{eq:homogeneity}
			\phi(\lambda X) = \omega(\lambda) \phi(X), \quad  \text{ for all } \lambda \in \F \text{ and } X \in M_n.
		\end{equation}
		Further, if $n \ge 2$, the map $\omega : \F \to \F$ is a ring monomorphism. In particular, $\phi$ is $\K$-homogeneous.
	\end{lemma}
	\begin{proof}
		In view of Lemmas \ref{le:zero map} and \ref{le:basic properties II} (c) and (e), $\phi(E_{11}), \ldots ,\phi(E_{nn})$ are mutually orthogonal rank-one idempotents and therefore can be simultaneously diagonalized. Hence, by passing to map $T^{-1}\phi(\cdot)T$, for a suitable invertible matrix $T \in M_n$, without loss of generality we can assume that 
		\begin{equation}\label{eq:phi fiksira dijagonalne}
			\phi(E_{jj})  = E_{jj} \quad \text{ for all } j \in [n].
		\end{equation}
		Obviously $\phi(\F^\times E_{jj}) \ne \{0\}$ (again by Lemma \ref{le:zero map}), for all $j \in [n]$. Note that for each $X \in M_n$ and $S \subseteq [n]$ we have
		\begin{equation}\label{eq:preserves support}
			\supp X \subseteq S \times S \implies \supp \phi(X) \subseteq S \times S.
		\end{equation}
		Indeed, denote the diagonal idempotent as
		$P := \sum_{j \in [n]\setminus S} E_{jj}$ and note that a matrix $X\in M_n$ is supported in $S \times S$ if and only if $XP=PX=0$. In that case, obviously  $X \circ P = 0$, so 
		$$0 = \phi(X \circ P) = \phi(X) \circ \phi(P) \stackrel{\text{Lemma }\ref{le:basic properties II} (g),\, \eqref{eq:phi fiksira dijagonalne}}= \phi(X) \circ P$$
		and hence Lemma \ref{le:Jordan product calculations} (a) implies the claim.
		\smallskip
		
		Let $j \in [n]$ and $\lambda \in \F^\times$. Then
		$$\phi(\lambda E_{jj}) = \phi((\lambda E_{jj}) \circ E_{jj}) = \phi(\lambda E_{jj}) \circ E_{jj}.$$
		In view of Lemma \ref{le:Jordan product calculations} (b) we have $$\phi(\lambda E_{jj}) = E_{jj}\phi(\lambda E_{jj})E_{jj} = \phi(\lambda E_{jj})_{jj} E_{jj}.$$
		Since $\phi(0) = 0$, it follows that there exists a unique map $\omega_j : \F \to \F$ such that $$\phi(\lambda E_{jj}) = \omega_j(\lambda)\phi(E_{jj}), \quad \text{ for all } \lambda \in \F.$$
		Fix distinct $i,j \in [n]$. For $\lambda \in \F^\times$ by \eqref{eq:phi fiksira dijagonalne} we have
		\begin{align*}
			\omega_i(2\lambda) \phi(E_{ij}) \circ E_{ii} &= \phi(E_{ij} \circ (2\lambda E_{ii}))  = \phi(\lambda E_{ij}) = \phi(E_{ij} \circ (2\lambda E_{jj}))\\
			&= \omega_j(2\lambda) \phi(E_{ij}) \circ E_{jj}.
		\end{align*}
		Note that \eqref{eq:preserves support} implies that $\supp \phi(\lambda E_{ij}) \subseteq \{i,j\}\times\{i,j\}$. By $\phi(\lambda E_{ij})^2 = 0$ and Lemma \ref{le:zero map} it follows that
		\begin{equation}\label{eq:Eij or Eji}
			\phi(\lambda E_{ij}) \propto E_{ij} \text{ or } E_{ji}.
		\end{equation}
		Returning to the previous equation, it follows that $\omega_i(2\lambda) = \omega_j(2\lambda)$. We conclude $\omega_i = \omega_j$ so there exists a unique globally defined map $\omega : \F \to \F$ such that
		$$\phi(\lambda E_{jj}) = \omega(\lambda)E_{jj}, \quad \text{ for all } \lambda \in \F, j \in [n].$$
		Now we prove \eqref {eq:homogeneity}. For $\lambda \in \F$ we have \begin{align*}
			\phi(\lambda I)&= \phi\left(\sum_{j \in [n]} \lambda E_{jj} \right)\stackrel{\text{Lemma } \ref{le:basic properties II} \text{ (h)}}= \sum_{j \in [n]} \phi(\lambda E_{jj})  =\sum_{j \in [n]}\omega(\lambda) \phi(E_{jj})\\
			&\leftstackrel{\eqref{eq:phi fiksira dijagonalne}}= \omega(\lambda) I.
		\end{align*}
		Now, for arbitrary $X \in M_n$ and $\lambda \in \F$ we have
		$$\phi(\lambda X) = \phi(X \circ (\lambda I)) = \phi(X) \circ \phi(\lambda I) = \omega(\lambda) \phi(X).$$
		For some $i \in [n]$ and $\lambda, \mu \in \F$ (again using \eqref{eq:phi fiksira dijagonalne}) we have
		\begin{align*}
			\omega(\lambda\mu)E_{ii} &= \phi((\lambda\mu) E_{ii}) = \phi((\lambda E_{ii}) \circ (\mu E_{ii})) = \phi(\lambda E_{ii}) \circ \phi(\mu E_{ii})\\
			&= \omega(\lambda)\omega(\mu)E_{ii},
		\end{align*}
		which implies $\omega(\lambda\mu) = \omega(\lambda)\omega(\mu)$, so $\omega$ is a multiplicative map.
		
		Assume now that $n \ge 2$. The argument that $\omega$ is additive is similar to the proof of \cite[Theorem~3.1, Claim~5]{GogicTomasevic2}. For completeness, we include the details. Let $i,j \in [n]$ be distinct. For fixed $x,y \in \F$ consider the idempotents
		$$E_{ii} + x E_{ij}, E_{jj} + y E_{ij} \in \idem(M_n).$$
		By \eqref{eq:preserves support} we see that
		$$\supp \phi(E_{ii} + x E_{ij}) \subseteq \{i,j\} \times \{i,j\}.$$
		Denote $$\phi (E_{ii} + x E_{ij}) = \sum_{(r,s) \in \{i,j\} \times \{i,j\}}\alpha_{rs} E_{rs}, \quad \alpha_{rs} \in \F.$$
		From now on, in view of \eqref{eq:Eij or Eji} assume that $\phi(E_{ij}) = \beta E_{ij}$ for some $\beta \in \F^\times$ as the other case (i.e.\ $\phi(E_{ij}) = \beta E_{ji}$) is similar. We have
		\begin{align*}
			\omega\left(\frac12 x\right) \beta E_{ij} \quad &= \phi\left(\frac12 x E_{ij}\right) = \phi((E_{ii} + x E_{ij}) \circ E_{jj})\stackrel{\eqref{eq:phi fiksira dijagonalne}}= \phi(E_{ii} + x E_{ij}) \circ E_{jj} \\
			&= 
			\frac12 \alpha_{ij}E_{ij} + \frac12 \alpha_{ji}E_{ji} + \alpha_{jj} E_{jj}.
		\end{align*}
		Since $\phi(E_{ii} + x E_{ij})$ is an idempotent and $\omega^{-1}(\{0\}) = \{0\}$, we conclude
		$$\alpha_{ij} =
		2 \omega\left(\frac12 x\right)  \beta, \qquad \alpha_{ji} = \alpha_{jj} = 0, \qquad \alpha_{ii} = 1.$$
		Hence
		$$\phi(E_{ii} + x E_{ij}) = E_{ii} + 2 \omega\left(\frac12 x\right) \beta E_{ij}.$$
		In an analogous way we arrive at the equality
		$$\phi(E_{jj} + y E_{ij}) = 
		E_{jj} + 2 \omega\left(\frac12 y\right)  \beta E_{ij}.$$
		We have
		\begin{align*}
			\omega\left(\frac{x+y}2\right) \beta E_{ij}\quad &= \phi\left(\frac{x+y}2 E_{ij}\right) = \phi((E_{ii} + x E_{ij}) \circ (E_{jj} + y E_{ij})) \\
			&= \phi(E_{ii} + x E_{ij}) \circ \phi(E_{jj} + y E_{ij}) \\
			&= \left(E_{ii} + 2 \omega\left(\frac12 x\right) \beta E_{ij}\right)  \circ \left(E_{jj} + 2 \omega\left(\frac12 y\right) \beta E_{ij}\right)\\
			&= \left(\omega\left(\frac12 x\right) + \omega\left(\frac12 y\right)\right)\beta E_{ij}
		\end{align*}
		and hence
		$$\omega\left(\frac{x+y}2\right) = \omega\left(\frac12 x\right) + \omega\left(\frac12 y\right).$$
		Since $x,y \in \F$ were arbitrary, this concludes the proof.
	\end{proof}
	The proof of the next lemma follows exactly the same lines as the proof of \cite[Theorem~3.1, Claim~8]{GogicTomasevic2}, so we omit it.
	\begin{lemma}\label{le:preserves triple product}
		Let $\phi : M_n \to M_n, n \ge 2$, be a $\circ$-multiplicative map such that $\phi(0) = 0$. Then
		$$\phi(PXP) = \phi(P)\phi(X)\phi(P), \quad \text{ for all }X \in M_n, P \in \idem(M_n).$$
	\end{lemma}
	
	\begin{proof}[Proof of Theorem \ref{thm:main}]
		First, as noted in the introduction (following the beginning of the proof of \cite[Theorem~3.1]{GogicTomasevic2}), it suffices to prove Theorem \ref{thm:main} for $\circ$-preserving maps, since the transformation \eqref{eq:supsitution Jmultiplicative} allows us to extend the result to $\diamond$-preserving maps. Therefore, assume that $\phi : M_n \to M_n$, $n \ge 2$, is $\circ$-multiplicative. 
		
		\smallskip
		
		Suppose that $\phi$ is not the zero map. Since $\phi(0)$ is an idempotent, without loss of generality we can assume that
		$$\phi(0) = \begin{bmatrix}
			I_r & 0 \\ 0 & 0
		\end{bmatrix}, \quad \text{ for some }0 \le r \le n.$$
		Assume that $r \ge 1$. Then we claim that $\phi$ is the constant map equal to $\phi(0)$. Indeed, first note that for all $X \in M_n$ we have
		$$\phi(0) = \phi(X \circ 0) = \phi(X) \circ \phi(0)$$
		which easily implies that 
		\begin{equation}\label{eq:X to Y}
			\phi(X) = \begin{bmatrix}
				I_r & 0 \\ 0 & \psi(X)
			\end{bmatrix},
		\end{equation}
		for some uniquely determined matrix $\psi(X) \in M_{n-r}$. In particular, if $r = n$, it follows that $\phi$ is the constant map globally equal to $I$. Otherwise, it makes sense to consider the map $\psi : M_n \to M_{n-r}$ defined by \eqref{eq:X to Y}, which is again  $\circ$-multiplicative, and satisfies $\psi(0) = 0$. Since $n-r < n$, Lemma \ref{le:basic properties II} (d) implies that $\psi$ must be the zero map, and therefore $\phi(X) = \phi(0)$ for all $X \in M_n$.
		
		Suppose now that $r = 0$, i.e.\ $\phi(0) = 0$. We claim that $\phi$ takes the form given in \eqref{eq:form of phi}, and as a result, it is a nonzero additive map. As in Lemma \ref{le:h exists} and its proof, without loss of generality we can assume that
		$$\phi(E_{jj}) = E_{jj}, \quad \text{ for all }j \in [n]$$
		and consequently, by Lemma \ref{le:basic properties II} (g),
		\begin{equation}\label{eq:identity on diagonal idempotents}
			\phi(P) = P, \quad\text{ for all } P \in \idem(M_n) \cap \ca{D}_n. 
		\end{equation}
		As in \eqref{eq:Eij or Eji} for $\lambda= 1$, under these assumptions we also obtain
		$$\phi(E_{ij}) \propto E_{ij} \text{ or }E_{ji}.$$
		We claim that the same option holds throughout.  We follow a similar approach as outlined in the proof of \cite[Theorem~3.1, Claim 4]{GogicTomasevic2}. For completeness we include the details. Consider the set
		$$\ca{S} := \{(r,s)\in [n]^2 \setminus \Delta_n : \phi(E_{rs}) \propto E_{rs}\}.$$  For the sake of concreteness, assume that $(i,j) \in \ca{S}$. In that case clearly $\phi(E_{ji}) \propto E_{ji}$, as otherwise
		$$\phi\left( \frac12 \left(E_{ii}+E_{jj}\right)\right) = \phi(E_{ij} \circ E_{ji}) = \phi(E_{ij}) \circ \phi(E_{ji}) \propto E_{ij} \circ E_{ij} = 0,$$
		which is a contradiction to Lemma \ref{le:zero map}, so $(j,i) \in \ca{S}$. The next objective is to show that $(i,k)\in\ca{S}$ for any $k \in [n]\setminus\{i\}$, and $(l,j)\in \ca{S}$ for any $l \in [n]\setminus\{j\}$.
		\begin{itemize}
			\item[--] Assume that $k \in [n]\setminus\{i,j\}$ and that $\phi(E_{ik}) \propto E_{ki}$. Then 
			$$0 = \phi(E_{ij} \circ E_{ik}) = \phi(E_{ij}) \circ \phi(E_{ik}) \propto E_{ij} \circ E_{ki} = \frac12 E_{kj},$$
			which is a contradiction, so it must be $\phi(E_{ik})  \propto E_{ik}$.
			\item[--] Assume that $l \in [n]\setminus\{i,j\}$ and that $\phi(E_{lj}) \propto E_{jl}$. Then 
			$$0 = \phi(E_{ij} \circ E_{lj}) = \phi(E_{ij}) \circ \phi(E_{lj}) \propto E_{ij} \circ E_{jl} = \frac12 E_{il},$$
			which is a contradiction, so that $\phi(E_{lj})  \propto E_{lj}$.
		\end{itemize}
		By Lemma \ref{le:contains entire CxC} it follows that $\ca{S} = [n]^2 \setminus \Delta_n$, so there exists a map $g : [n]^2 \to \F^\times$ such that
		$$\phi(E_{ij}) = g(i,j) E_{ij}, \quad \text{ for all }i,j \in [n].$$ We claim that the map $g$ is transitive in the sense of \cite{Coelho}, i.e.\ it satisfies
		$$g(i,j)g(j,k) = g(i,k), \quad \text{ for all }(i,j), (j,k) \in [n]^2.$$ Fix $(i,j), (j,k) \in [n]^2$. If $i \ne k$, then
		\begin{align*}
			\frac12 g(i,k)E_{ik} &\stackrel{\text{Lemma }\ref{le:h exists}}= \phi\left(\frac12 E_{ik}\right) =\phi(E_{ij} \circ E_{jk}) = \phi(E_{ij}) \circ \phi(E_{jk}) \\
			&= \frac12 g(i,j)g(j,k)E_{ik},
		\end{align*}
		which implies $g(i,k) = g(i,j)g(j,k)$.
		On the other hand, if $i = k$, then
		\begin{align*}
			\frac{1}{2}\left(E_{ii} + E_{jj}\right) \qquad &\leftstackrel{\text{Lemma }\ref{le:h exists}, \eqref{eq:identity on diagonal idempotents}}= \phi\left(\frac{1}{2}\left(E_{ii} + E_{jj}\right)\right) =\phi(E_{ij} \circ E_{ji}) = \phi(E_{ij}) \circ \phi(E_{ji}) \\
			&= \frac12 g(i,j)g(j,i)(E_{ii} + E_{jj})
		\end{align*}
		which implies $g(i,i) = 1 = g(i,j)g(j,i)$. Following \cite{Coelho}, denote by $$g^* : M_n \to M_n, \qquad g^*(E_{ij}) := g(i,j)E_{ij}$$
		the induced (algebra) automorphism of $M_n$. Since every automorphism of $M_n$ is inner (see e.g.\ \cite[Theorem~1.30]{Bresar-book}), by passing to the map $X \mapsto (g^*)^{-1}(\phi(X))$, without loss of generality we can assume that 
		$$\phi(E_{ij}) = E_{ij}, \quad \text{ for all }  (i,j) \in [n]^2.$$
		In view of Lemma \ref{le:h exists}, denote by $\omega :\F \to \F$ the induced ring monomorphism that satisfies \eqref{eq:homogeneity}. We claim that 
		$$\phi(X) = \omega(X), \quad \text{ for all } X \in M_n.$$
		Fix $(i,j) \in [n]^2$. If $i = j$, we have
		\begin{align*}
			\omega(X_{ii})E_{ii} = \phi(X_{ii}E_{ii}) = \phi(E_{ii}XE_{ii}) \stackrel{\text{Lemma }\ref{le:preserves triple product}}= E_{ii}\phi(X)E_{ii} = \phi(X)_{ii}E_{ii},
		\end{align*}
		so $\phi(X)_{ii} = \omega(X_{ii})$. Now assume $i \ne j$. Since $\omega$ is multiplicative and acts as the identity on the prime subfield $\K \subseteq \F$, we obtain
		\begin{align*}
			\frac12 \omega(X_{ij})E_{ji} &= \phi\left(\frac12 X_{ij} E_{ji}\right) = \phi\left(\frac12 E_{ji}XE_{ji}\right) = \phi((E_{ji} \circ X ) \circ E_{ji})\\
			&= (\phi(E_{ji}) \circ \phi(X)) \circ \phi(E_{ji}) = (E_{ji} \circ \phi(X)) \circ E_{ji}\\
			&= \frac12 \phi(X)_{ij}E_{ji}.
		\end{align*}
		This implies $\phi(X)_{ij} = \omega(X_{ij})$, which completes the proof of the theorem.
	\end{proof}
	It is also worth noting that the first part of the proof of Theorem \ref{thm:main} immediately yields  the following corollary:
	\begin{corollary} 
		Let $m < n$. The map $\phi : M_n \to M_m$ is Jordan multiplicative if and only if it is constant and equal to a fixed idempotent. 
	\end{corollary}
	In contrast to the matrix algebra $M_n$, the next two examples illustrate that non-constant Jordan multiplicative maps defined on general central SMAs (i.e.\ those with a trivial centre), or on the $C^*$-algebra $\BB(H)$ of bounded linear operators on an infinite-dimensional Hilbert space $H$, are no longer automatically additive.  
	\begin{example}\label{ex:notSMA}
		Let $\ca{T}_n \subseteq M_n$ be the upper-triangular subalgebra of $M_n$. Choose an arbitrary non-additive multiplicative map $\omega : \F \to \F$ (e.g.\ $\omega(x):=x^2$) and define a map 
		$$
		\phi: \ca{T}_n \to \ca{T}_n, \qquad \begin{bmatrix} x_{11} & x_{12} & \cdots & x_{1n} \\ 0 & x_{22} & \cdots & x_{2n} \\
			\vdots & \vdots & \ddots & \vdots \\
			0 & 0 & \cdots & x_{nn} \end{bmatrix} \mapsto \begin{bmatrix} \omega(x_{11}) & 0 & \cdots & 0 \\
			0 & \omega(x_{22}) & \cdots & 0 \\
			\vdots & \vdots & \ddots & \vdots \\
			0 & 0 & \cdots &\omega(x_{nn}) \end{bmatrix}.
		$$
		Then $\phi$ is clearly $\circ$-multiplicative, but is neither constant nor additive.  
	\end{example}
	
	\begin{example}\label{ex:notB(H)}
		Let $H$ be an infinite-dimensional Hilbert space. In view of the identification $H \cong H \oplus H$, for any fixed nonzero idempotent $P \in \BB(H)$, the map
		$$\phi : \BB(H) \to \BB(H \oplus H), \qquad X \mapsto \begin{bmatrix}
			X & 0 \\ 0 & P
		\end{bmatrix}$$
		is $\circ$-multiplicative, but is neither constant nor additive.
	\end{example}
	Finally, the next simple example  demonstrates that Theorem \ref{thm:main} does not extend to $\diamond$-multiplicative maps over fields $\F$ of characteristic two.
	\begin{example}\label{ex:notchar2}
		Assume that $\chr(\F)=2$ and $n \ge 2$. Fix an arbitrary trace-one matrix $A \in M_n(\F)$ and define a map $\phi : M_n(\F) \to M_n(\F)$ that sends $A$ to a fixed nonzero matrix and all other matrices to zero. As the trace of any matrix in $M_n(\F)$ of the form $X\diamond Y=XY+YX$ is zero, it follows that $\phi$ is $\diamond$-multiplicative, but is neither constant nor additive. 
	\end{example}

\end{document}